\documentclass[a4paper,11pt]{article}
\usepackage{mathrsfs}
\usepackage{}
\usepackage[top=2.5cm,bottom=2.5cm,left=2.5cm,right=2.5cm]{geometry}
\usepackage{amssymb}
\usepackage{graphicx}
\usepackage{amsmath,amsthm,amssymb,lineno}
\usepackage{latexsym}
\usepackage{epstopdf}
\usepackage{setspace}
\usepackage{graphicx,booktabs,multirow}
\usepackage{latexsym, tabularx,shapepar}
\usepackage[all,2cell,dvips]{xy} \UseAllTwocells \SilentMatrices
\usepackage{appendix}
\usepackage{longtable}
\usepackage{cite}
\usepackage{CJK}
\usepackage{indentfirst}
\usepackage{array}
\usepackage{amsmath}
\allowdisplaybreaks[4]
\graphicspath{{figures/}}

\newtheorem{theorem}{Theorem}[section]
\newtheorem{lemma}[theorem]{Lemma}
\newtheorem{corollary}[theorem]{\rm\bfseries Corollary}

\newtheorem{remark}[theorem]{Remark}
\begin{document}

\title{
Extremed signed graphs for triangle
}
\author{  Dijian Wang$^{a}$,    Yaoping Hou$^{b}$\thanks {Corresponding author: yphou@hunnu.edu.cn},  Deqiong Li$^{c}$ \\
\small  $^{a}$School of Science, Zhejiang University of Science and Technology, \\
\small Hangzhou, Zhejiang, 310023, P. R. China\\
\small $^{b}$College of Mathematics and Statistics, Hunan Normal University,\\
\small Changsha,  Hunan, 410081, P. R. China\\
\small $^c$School  of Mathematics and Statistics, Hunan  University of  Technology and Business, \\
\small  Changsha, Hunan, 410205, P. R. China\\}
\date{}
\maketitle
\begin{abstract}
In this paper, we study the  Tur\'{a}n  problem of signed graphs version.
 Suppose that $\dot{G}$ is a connected unbalanced signed graph of order $n$ with   $e(\dot{G})$ edges and  $e^-(\dot{G})$ negative edges, and let $\rho(\dot{G})$ be the spectral radius of $\dot{G}.$  
 The signed graph $\dot{G}^{s,t}$ ($s+t=n-2$)  is obtained from an all-positive clique $(K_{n-2},+)$ with  $V(K_{n-2})=\{u_1,\dots,u_s,v_1,\dots,v_t\}$ ($s,t\ge 1$)
 and two isolated vertices $u$ and $v$ by   adding negative edge $uv$ and  positive edges $uu_1,\dots,uu_s,vv_1,\dots,vv_t.$
  Firstly, we prove that if $\dot{G}$ is $C_3^-$-free, then $e(\dot{G})\le \frac{n(n-1)}{2}-(n-2),$ with equality holding if and only if $\dot{G}\sim \dot{G}^{s,t}.$ 
  Moreover, $e^-(\dot{G}^{s,t})\le \lfloor\frac{n-2}{2}\rfloor\lceil\frac{n-2}{2}\rceil+n-2,$ with equality holding if and only if $\dot{G}^{s,t}= \dot{G}_U^{\lfloor\frac{n-2}{2}\rfloor,\lceil\frac{n-2}{2}\rceil},$ where  $\dot{G}_U^{\lfloor\frac{n-2}{2}\rfloor,\lceil\frac{n-2}{2}\rceil}$ is obtained from $\dot{G}^{\lfloor\frac{n-2}{2}\rfloor,\lceil\frac{n-2}{2}\rceil}$
  by switching at vertex set $U=\{v,u_1,\dots,u_{\lfloor\frac{n-2}{2}\rfloor}\}.$
  Secondly, we prove that if $\dot{G}$ is $C_3^-$-free, then $\rho(\dot{G})\le \frac{1}{2}( \sqrt{ n^2-8}+n-4),$ with equality holding if and only if $\dot{G}\sim \dot{G}^{1,n-3}.$
\end{abstract}

\baselineskip=0.30in

\section {Introduction}
In this paper, we consider only simple, undirected and finite graphs $G$ with $n$ vertices and $e(G)$ edges.
 For  a graph $G=(V(G),E(G)),$
let $A(G)$ denote the adjacency matrix of $G,$
the eigenvalues of $A(G)$ are all real and the spectral radius  of $G$  is equal to the  largest eigenvalue of $A(G)$.
Signed graphs $\dot{G}=(G,\sigma)$ are graphs $G$ whose edges get signs $+1$ or $-1,$  where $G$ is called \emph{underlying graph} and $\sigma: E(G)\rightarrow  \{+1,-1\}$ is a 
sign function.
The edge $uv$ is  \emph{positive} (\emph{negative}) if $uv$ gets  sign $+1$ (resp. $-1$)
and denoted by $u\mathop{\sim}\limits^{+} v$ (resp. $u\mathop{\sim}\limits^{-} v$).  Signed graph $\dot{G}$ is called  \emph{homogeneous} if 
all  edges get same sign,  and \emph{heterogeneous} otherwise.
If all edges  get signs $+1$ (resp. $-1$), then $\dot{G}$ is called \emph{all-positive} (resp. \emph{all-negative}) and  denoted by $(G,+)$  (resp. $
(G,-)$).

Let $H$ be a simple graph. A graph $G$ is $H$-free if there is no subgraph of $G$ isomorphic to $H$.
 The \emph{Tur$\acute{a}$n number}, $ex(n,H),$ is the maximum number of edges in a $H$-free graph of order $n.$ 
 The  \emph{Tur\'{a}n problem} is to
  determine the exact value of $ex(n,H),$ which is a central problem of extremal graph theory. 
 However, there are only a few cases when the  \emph{Tur$\acute{a}$n number} is known. For more details on this topic, the readers may be referred to see \cite{T39} for  $ex(n,K_{r}),$ \cite{B1971,B71,F15} for $ex(n,C_{2k+1}),$ and \cite{B74,F96,F06} for $ex(n,C_{2k}).$

Recently, many researches 
 focused on spectral analogues of the Tur\'{a}n type problem for graphs, which was proposed by Nikiforov \cite{N09}. The spectral Tur\'{a}n type problem is to determine the maximum spectral radius instead of the number of edges among all $n$-vertex $H$-free graphs. The graph which attains the maximum spectral radius is called a spectral extremal graph. 
 In the past few decades, much attention has been paid to search for the spectral extremal graph for various families of graphs $H$ such as $\{K_r\}$ \cite{N07,W86}, $\{K_{s,t}\}$ \cite{B09,N07} and $\{C_4,C_6\}$ \cite{N07,Z12,Z20}. 
 
The \emph{adjacency matrix} of  $\dot{G}$ is defined by $A_\sigma=(\sigma_{ij}),$ where 
$\sigma_{ij} =\sigma(v_iv_j)$ if $v_i \sim v_j,$ and $\sigma_{ij} = 0$ otherwise.
The eigenvalues of $\dot{G}$ are identified to be the eigenvalues of $A_\sigma$,
denoted by $\lambda_1\geq  \lambda_2\geq  \dots\geq \lambda_n$. 
The spectral radius of $\dot{G}$ is usually defined by $\rho(\dot{G})=max\{|\lambda_i|:1\le i\le n\}=max\{\lambda_1,-\lambda_n\}.$
The sign of a cycle $C$ of $\dot{G}$ is  $\sigma(C)=\prod_{e\in C}\sigma (e)$, whose sign is $+1$ (resp. $-1$) is called \emph{positive} (resp. \emph{negative}). A signed graph $\dot{G}$ is called \emph{balanced} if  all its cycles are positive; otherwise it is called \emph{unbalanced.}

An important feature of signed graphs is the concept of \emph{switching} the signature.   For $U\subset V (\dot{G}),$ let $\dot{G}_U=(G,\sigma_U)$ be the signed graph obtained from $\dot{G}=(G,\sigma)$ by reversing the signature of the edges in the cut $[U, V (G) \setminus U], $
namely $\sigma_U(e) = -\sigma(e)$ for any edge $e$ between $U$ and $V(G) \setminus U,$ and $\sigma_U(e) = \sigma(e)$ otherwise. We say that  two signed graphs $\dot{G}$ and $\dot{G}_U$ are   switching equivalent,  write $\dot{G} \sim \dot{G}_U$. 
Furthermore, it is important to observe that switching equivalent signed graphs have similar adjacency matrices. In fact, any switching arising from $U$ can be realized by a diagonal matrix $S_U =diag(s_1,s_2,\dots,s_n)$ having $s_i = 1$ for each $i\in U$, and $s_i = -1$ otherwise. Hence, $A_\sigma = S_U A_{\sigma_U}S_U$; in this case we say that the matrices are signature similar. When we consider a signed graph $\dot{G}$, from a spectral viewpoint, we always consider its switching equivalent class $[\dot{G}]$. From the eigenspace viewpoint, the eigenvector components are also switched in signs. Evidently, for each eigenvector, there exists a suitable switching such that all components become nonnegative.
 For more basic results on signed graphs, we refer to Zaslavsky 
\cite{Z82,Z08}.

Signed graphs are natural generalizations of (unsigned) graphs. Therefore, spectral problems defined and studied for unsigned graphs can be considered in terms of signed graphs, and sometimes such generalization shows nice properties which cannot be appreciated in terms of (unsigned) graphs.  Such as Huang \cite{H19} studied the spectral properties of the  adjacency matrix of signed hypercubes $(Q_n,\sigma)$ and solved Sensitivity Conjecture.
In 2018,  Belardo, Cioab$\breve{a}$, Koolen and Wang \cite{B18} surveyed some general results on the  spectra of signed graphs and one of  results is that: for  a signed graph $\dot{G}=(G,\sigma),$ then $\rho(\dot{G})\le \rho(G)$. 
Because of this result, when we study the spectral radius of signed graph, we always  exclude the balanced case and
reduce to the unbalanced signed graphs. In 2019, Akbari,  Belardo,  Heydari,  Maghasedi and  Souri \cite{A19} determined signed graphs achieving the minimal or the maximal index in the class of unbalanced unicyclic graphs.
In 2021, He, Li, Shan and Wang \cite{H21} determined the first five largest indices among all unbalanced signed bicyclic graphs.  In 2022, Brunetti and Stani\'{c} \cite{B22} studied the extremal spectral radius in all unbalanced connected signed graphs. 

Note that the signed triangle $(C_3,\sigma)$ has only two switching equivalences: one is all-positive (denoted by  $C_3^+$) and another contains one negative edge (denoted by $C_3^-$).
In this paper, we study the (spectral) Tur\'{a}n problem in a $C_3^-$-free or a $C_3^+$-free signed graph of order $n$.

For a simple  graph $G,$ it is known that $e(G)\le \frac{n(n-1)}{2},$ with equality holding if and only if $G$ is a complete graph. Up to switching equivalence, we know that a complete signed graph $\dot{G}=(K_n,\sigma)$ is $C_3^+$-free (resp. $C_3^-$-free) if and only if  $\dot{G}\sim (K_n,-)$ (resp. $\dot{G}\sim (K_n,+)$). Then   the following result holds.

\begin{theorem}\label{p1}
Let $\dot{G}=(G,\sigma)$  be a connected signed graph with  $n$ vertices and $e(\dot{G})$ edges.

$(i)$ If  $\dot{G}$ is  $C_3^+$-free, then $e(\dot{G})\le \frac{n(n-1)}{2},$ with equality holding if and only if $\dot{G}\sim (K_n,-).$

$(ii)$ If  $\dot{G}$ is  $C_3^-$-free, then $e(\dot{G})\le \frac{n(n-1)}{2},$ with equality holding if and only if $\dot{G}\sim (K_n,+).$
\end{theorem}
 Obviously, Theorem \ref{p1} is trivial.
 Therefore, we should focus our attention on  the heterogeneous and unbalanced signed graphs. Let $\dot{G}$ be a  connected unbalanced signed graph with  $n$ vertices and $e(\dot{G})$ edges.  If $\dot{G}$ is  $C_3^+$-free, 
  we can see that the all-negative complete signed graph $(K_n,-)$ attains the   maximum number of edges (see Theorem \ref{p1}) and  the maximum  spectral radius (see \cite[Theorem 3.1]{B22}). Hence here we mainly consider that $\dot{G}$ is  $C_3^-$-free. 
 This paper we will determine  signed graphs of order $n$ that have the maximum number of edges and  the maximum  spectral radius among all unbalanced signed graphs containing no $C_3^-.$

For two   vertex-disjoint graphs $G$ and $H,$
 $G\cup H$ denotes the disjoint union of $G$ and $H$, and  $G\vee H$ denotes the join of $G$ and $H,$ which is obtained from  $G\cup H$ by adding all possible edges between $G$ and $H.$
The signed graph $\dot{G}^{s,t}$ ($s+t=n-2$ and $s,t\ge 1$)  is obtained from an all-positive clique $(K_{n-2},+)$ with  $V(K_{n-2})=\{u_1,\dots,u_s,v_1,\dots,v_t\}$
 and two isolated vertices $u$ and $v$ by   adding  negative edge $uv$ and  positive edges $uu_1,\dots,uu_s,vv_1,\dots,vv_t.$ See Fig. \ref{f1}.
Our main results are reported as followings.

\begin{theorem}\label{thm1}
Let $\dot{G}=(G,\sigma)$  be a connected unbalanced  signed graph of order $n$ with   $e(\dot{G})$ edges and  $e^-(\dot{G})$ negative edges.
 If  $\dot{G}$ is  $C_3^-$-free, then  $$e(\dot{G})\le \frac{n(n-1)}{2}-(n-2),$$ with equality holding if and only if $\dot{G}\sim \dot{G}^{s,t},$ where $s+t=n-2$ and $s,t\ge 1$.  Moreover, $$e^-(\dot{G}^{s,t})\le \lfloor\frac{n-2}{2}\rfloor\lceil\frac{n-2}{2}\rceil+n-2,$$ with equality holding if and only if $\dot{G}^{s,t}=\dot{G}_U^{\lfloor\frac{n-2}{2}\rfloor,\lceil\frac{n-2}{2}\rceil},$ where  $\dot{G}_U^{\lfloor\frac{n-2}{2}\rfloor,\lceil\frac{n-2}{2}\rceil}$ is obtained from $\dot{G}^{\lfloor\frac{n-2}{2}\rfloor,\lceil\frac{n-2}{2}\rceil}$
  by switching at vertex set $U=\{v,u_1,\dots,u_{\lfloor\frac{n-2}{2}\rfloor}\}.$ See Fig. \ref{f1}.
\end{theorem}

\begin{theorem}\label{thm2}
Let $\dot{G}=(G,\sigma)$  be  a connected  unbalanced signed graph of order $n$.
 If  $\dot{G}$ is  $C_3^-$-free, then  $$\rho(\dot{G})\le \frac{1}{2}( \sqrt{ n^2-8}+n-4),$$ with equality holding if and only if $\dot{G}\sim \dot{G}^{1,n-3}.$
\end{theorem}

\section{Proof of Theorem \ref{thm1}}

   \begin{figure}
\begin{center}
  \includegraphics[width=14cm,height=3cm]{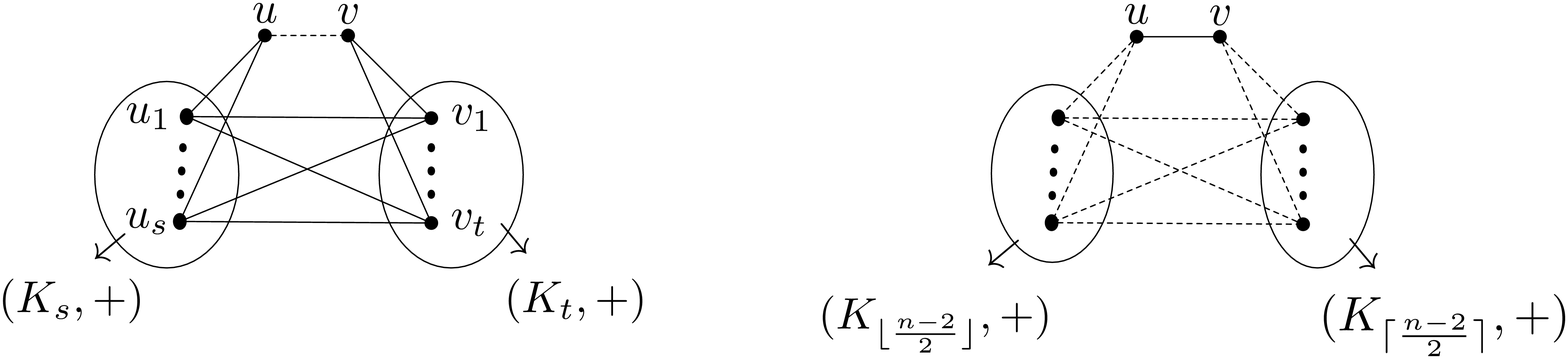}
  \end{center}
   \vskip -0.6cm\caption{The signed graphs $\dot{G}^{s,t}$ and $\dot{G}_U^{\lfloor\frac{n-2}{2}\rfloor,\lceil\frac{n-2}{2}\rceil}$ (dashed lines indicate negative edges).}
  \label{f1}
\end{figure}

Let $\dot{G}$ be a  connected unbalanced signed graph  of order $n$ with   $e(\dot{G})$ edges and  $e^-(\dot{G})$ negative edges.
Since $\dot{G}$ is unbalanced, then $\dot{G}$ contains at least one unbalanced cycle $\mathcal{C}$.
If there are more than one unbalanced cycles,  we choose
one for which the length of $\mathcal{C}$ is minimal. 
If $n=4,$ then $\dot{G}$ is the unbalanced 4-cycle and Theorem \ref{thm1} is trivial. Then we assume that $n\ge 5.$

The first lemma is used to check whether two signed graphs are switching equivalent.
\begin{lemma}\cite[Lemma 3.1]{Z82}\label{lemma2.1}
Let $G$ be a connected graph and $T$ a spanning tree of $G.$ The each switching equivalence class of signed graphs on the graph $G$ has a unique representative which is $+1$ on $T.$ Indeed, given any prescribed sign function $\sigma_T: T \to \{+1,-1\},$ each switching class has a single representative which agrees with $\sigma_T$ on $T.$
\end{lemma}

The set of   neighbours of a vertex $v$ (in $\dot{G}$) is denoted by $N(v).$ Then $N[v]=N(v)\cup \{v\}.$ The \emph{degree}
of a vertex $v$ (denoted by $d_v$) is the cardinality of $N(v)$. 
 For $U \subseteq V (G),$ let $\dot{G}[U]$ denote the  induced
subgraph by $U,$ which is of course a signed graph.

   Note that $\dot{G}$ is $C_3^-$-free and unbalanced, then we have the following two fact.

\noindent
\textbf{Fact 1.} For each vertex $v$ of $V(\dot{G})$, if 
all edges $vv_i$ get same signs 
 for each vertex $v_i\in N(v)$, then $\dot{G}[N(v)]$ is all-positive.

\noindent
\textbf{Fact 2.} For each vertex $v$ of $V(\dot{G})$, then $d_v\le n-2$.

Before giving the proof of Theorem \ref{thm1},
 we provide two lemmas.
 
  Since  $\mathcal{C}$ is unbalanced, then
 it  has at least  one negative edge, say $uv$. 
Let $N(u)=\{u_1,\dots,u_{s}\}$ and $N(v)=\{v_1,\dots,v_{t}\}$, then $s\ge 1$ and $t\ge 1,$ otherwise 
the edge $uv$  is not in any cycle.  By Lemma \ref{lemma2.1}, we can assume that $\sigma(uu_i)=+1$ for all $i=1,\dots,s.$ Let $\mathcal{C}=uu_iv_jvu$ (if the length of $\mathcal{C}$ is 4) or $\mathcal{C}=uu_iz_1\dots z_q v_jvu$ (if the length of $\mathcal{C}$ is greater than 4).

\begin{lemma}\label{2}
Let  $\mathcal{C}=uu_iv_jvu$ (if the length of $\mathcal{C}$ is 4) or $\mathcal{C}=uu_iz_1\dots z_q v_jvu$ (if the length of $\mathcal{C}$ is greater than 4) be the unbalanced cycle in $\dot{G}$. Then 
 $u_i\in N[u]\setminus N[v]$ and  
 $v_j\in N[v]\setminus N[u]$.
\end{lemma}
\begin{proof}
If the vertex $u_i$ in $\mathcal{C}$ belongs to  $N(u)\cap N(v),$ then 
$u_i\mathop{\sim}\limits^{-} v$ (as $\dot{G}$ is $C_3^-$-free) and
 we can find an unbalanced triangle  $\mathcal{C}=vu_iv_jv$ or an
  unbalanced cycle $\mathcal{C}^\prime=vu_iz_1\dots z_qv_jv$ with fewer length, which 
contradicts to the choice of $\mathcal{C}$. Thus, $u_i\in N[u]\setminus N[v].$ Similarly,  we have
 $v_j\in N[v]\setminus N[u].$
 \end{proof}

\begin{lemma}\label{lemma2.2}
For a connected $C_3^-$-free complete signed   graph $\dot{G}$, then $e^-(\dot{G})\le \lfloor\frac{n}{2}\rfloor\lceil\frac{n}{2}\rceil,$
with equality holding if and only if $\dot{G}$ is obtained  from $(K_n,+)$
  by switching at  $\lfloor\frac{n}{2}\rfloor$ vertices.
\end{lemma}
\begin{proof}
Let $V(\dot{G})=\{v\}\cup V^+\cup U^-,$ where $\sigma(vv_i)=+1$ for each vertex $v_i\in V^+$ and  
$\sigma(vu_j)=-1$ for each vertex $u_j\in U^-.$
Since $\dot{G}$ is $C_3^-$-free, then $\sigma(v_iv_j)=+1$ for each edge $v_iv_j\in E(V^+),$ 
$\sigma(u_iu_j)=+1$ for each edge $u_iv_j\in E(U^-)$ and  $\sigma(v_iu_j)=-1$ for each edge $v_iu_j$. 
Then  $e^-(\dot{G})\le (|V^+|+1)|U^-| \le \lfloor\frac{n}{2}\rfloor\lceil\frac{n}{2}\rceil,$
with equality holding if and only if $\dot{G}$ is obtained  from $(K_n,+)$
  by switching at  $\lfloor\frac{n}{2}\rfloor$ vertices.
\end{proof}

\noindent\emph{Proof of Theorem \ref{thm1}:}
 Suppose that $e(\dot{G})=\frac{n(n-1)}{2}-\ell$. 
Now we  count that how many edges  are lost in  signed graph $\dot{G}$.
 That is to consider the value of $\ell.$ Since $e(\dot{G}^{s,t})=\frac{n(n-1)}{2}-(n-2),$ we may suppose that $\ell \le n-2.$
Let the unbalanced cycle $\mathcal{C}$ and the negative $uv$ in $\mathcal{C}$ be defined as before.
Now we divide the proof into two cases.

\textbf{Case 1.} $N(u)\cap N(v)=\emptyset.$  Then $s+t\le n-2,$  with equality holding if and only if each vertex  in $\dot{G}$ is adjacent to either $u$ or $v$. Up to switching equivalence, we  let $\sigma(vv_i)=+1$ for all $i=1,\dots,t.$
Note that
 $$e(\dot{G}[V(G)\setminus \{u,v\}])\le \frac{(n-2)(n-3)}{2},$$
 with equality holding if and only if   $\dot{G}[V(G)\setminus \{u,v\}]$ is a clique, then
 \begin{equation}
\begin{split}e(\dot{G})&\le d_{u}+d_v-1+e(\dot{G}[V(G)\setminus \{u,v\}])\\
&=s+t+1+e(\dot{G}[V(G)\setminus \{u,v\}])\\
&\le n-1+ \frac{(n-2)(n-3)}{2}\\
&=\frac{n(n-1)}{2}-(n-2),\end{split}
\end{equation}
with equality holding if and only if $s+t= n-2$ and $e(\dot{G}[V(G)\setminus \{u,v\}])= \frac{(n-2)(n-3)}{2}$ if and only if  each vertex  in $\dot{G}$ is adjacent to either $u$ or $v$ and $\dot{G}[V(G)\setminus \{u,v\}]$ is a clique. 
Since $\dot{G}$ is unbalanced and $C_3^-$-free, then $\dot{G}[V(G)\setminus \{u,v\}]$ is an all-positive clique. Hence, 
$\dot{G}\sim \dot{G}^{s,t}.$

\textbf{Case 2.} $N(u)\cap N(v)\ne \emptyset.$ Suppose that $N(u)\cap N(v)=\{w_1,\dots,w_k\}$ $(k\ge 1)$,
then $s+t-k\le n-2$. Let $N[u]\setminus N[v]=\{u_1,\dots,u_{s_1}\}$ ($s_1+k=s$) and $N[v]\setminus N[u]=\{v_1,\dots,v_{t_1}\}$ ($t_1+k=t$), where $s_1\ge 1$ and $t_1\ge 1$ by Lemma \ref{2}.  Since $\dot{G}$ is  $C_3^-$-free, then $\sigma(vw_i)=-1$ for  $i=1,\dots,k.$ Up to switching equivalence, we can let $\sigma(vv_i)=-1$ for 
each $v_i\in N[v]\setminus N[u].$ Now switching at vertex $v,$ then
 all edges associated to the vertex $v$ are positive.
  
  We now give the following claim.

\noindent\textbf{Claim 1.} Each vertex in $\dot{G}$ is adjacent to at least one vertex of $u$ and $v.$
\begin{proof}
Suppose that $Z=\{z_1,\dots,z_p\}$ is a vertex subset of $V(\dot{G})$ such that each vertex $z_i$ is 
 adjacent to neither  $u$ nor $v.$ Then $2+s_1+t_1+k+p=n.$
Recall that $\mathcal{C}=uu_iv_jvu$ or  $\mathcal{C}=uu_iz_1\dots z_qv_jvu$.
  If $\sigma(u_iv_j)=-1$, since $\dot{G}$ is   $C_3^-$-free, then $u_i$ and $v_j$ have no common neighbor in $N(u)\cap N(v).$  If $\sigma(u_iv_j)=+1$ or 0, then 
  the length of $\mathcal{C}$ is at least 5,  
   $u_i$ and $v_j$ also have no common neighbor in $N(u)\cap N(v)$,  otherwise there is a vertex $w_{k_1}\in N(u)\cap N(v)$ with $w_{k_1}\mathop{\sim}\limits^{+}  u_i$ and $w_{k_1}\mathop{\sim}\limits^{+}  v_j$, then  we  find an  unbalanced cycle $\mathcal{C}^\prime=w_{k_1}u_iz_1\dots z_qv_jw_{k_1}$  with fewer length, which 
contradicts to the choice of $\mathcal{C}$.
   And now  we  obtain that $\ell\ge s_1+t_1+k+2p=n-2+p\ge n-1,$ which contradicts the hypothesis. Hence no such  $Z$  exists and each vertex in $\dot{G}$ is adjacent to at least one vertex of $u$ and $v.$
\end{proof}

By Claim 1, then  $V(\dot{G})=\{u,v\}\cup N(u)\cup N(v),$ $2+s_1+t_1+k=n$ and $u_{i}\mathop{\sim}\limits^{-}  v_j.$ Without loss of generality, assume that 
$\mathcal{C}=uu_1v_1v$ and
$\sigma(u_1v_1)=-1.$ 
If $N(u_1)\cap N(v_1)= \emptyset,$ it is  similar to  Case 1.
 We next suppose that $N(u_1)\cap N(v_1)\ne \emptyset$ and
 $u_2\in N(u_1)\cap N(v_1)$, then $u_{2}\mathop{\sim}\limits^{+}  u_1$ and $u_{2}\mathop{\sim}\limits^{-}  v_1.$
 Since $\dot{G}$ is $C_3^-$-free, then 
 each vertex $w_i$ ($i=1,2,\dots,k$) is adjacent to at most one vertex of $u_1$ and $v_1.$
Up to now, we find that $$\ell\ge s_1+t_1+k=n-2.$$
In addition, 
we  obtain that each vertex $w_i$ ($i=1,2,\dots,k$) is adjacent to exactly one vertex of $u_1$ and $v_1,$ otherwise $w_i\not\sim u_1, v_1$, then $\ell\ge n-1,$ which contradicts the hypothesis.
Similarly, we get that each vertex $w_i$ is adjacent to  exactly one vertex of $u_2$ and $v_1.$

Let $W=W_1\cup W_2,$ where   each vertex in $W_1$ (resp. $W_2$) is adjacent to the vertex  $u_1$ (resp. $v_1$), $|W_i|=k_i$ for $i=1$ or 2, and $k_1+k_2=k$.  Then $\ell\ge s_1+t_1+2k_2+k_1= n-2+k_2.$ Since $\ell \le n-2,$ then $k_2=0,$
which means that $v_1\not\sim w_i$  and $u_1\sim w_i$ for all $i=1,\dots,k.$

\textbf{Subcase 2.1.} $t_1=1.$ Then $N(v)\cap N(v_1)= \emptyset,$  which is  similar to  Case 1.

\textbf{Subcase 2.2.} $t_1\ge 2.$ Then $v_2$ is adjacent to at most two vertices of $\{w_1,v_1,u_2\}$ (as $\dot{G}$ is $C_3^-$-free). And we  have   $\ell\ge n-1$,  which contradicts the hypothesis.

From the above two cases, we conclude that  if  $\dot{G}$ is  $C_3^-$-free, then  $e(\dot{G})\le \frac{n(n-1)}{2}-(n-2),$ with equality holding if and only if $\dot{G}\sim \dot{G}^{s,t}.$

Let $E(\dot{G}^{s,t})=E_1\cup E_2$, where $E_1=\{uv,uu_1,\dots,uu_s,vv_1,$ $\dots,v_t\}$ and 
$E_2=E(\dot{G}^{s,t})\setminus E_1.$
Then $e^-(\dot{G}^{s,t})=e^-(E_1)+e^-(E_2).$ If $e^-(E_1)=e(E_1)=n-1,$ since $\dot{G}$ is $C_3^-$-free and unbalanced,  we can see that  $\dot{G}^{s,t}[V(\dot{G}^{s,t})\setminus \{u,v\}]$ is all-positive and $e^-(E_2)=0$. Then $e^-(\dot{G}^{s,t})=n-1< \lfloor\frac{n-2}{2}\rfloor\lceil\frac{n-2}{2}\rceil+n-2$. If $e^-(E_1)\le n-2,$ then  
 \begin{equation}\label{equation2}
 e^-(\dot{G}^{s,t})=e^-(E_1)+e^-(E_2) \le n-2+ \lfloor\frac{n-2}{2}\rfloor\lceil\frac{n-2}{2}\rceil ~~\text{(by Lemma \ref{lemma2.2})}.
\end{equation}
If $e^-(E_1)=n-2$ and $\sigma(uu_i)=+1$ (or $\sigma(vv_i)=+1$) for one $i,$ then
$\dot{G}^{s,t}$ is switching equivalent to the signed graph $\dot{G}_{U_1}^{s,t}$ by  switching at $U_1=\{u_i\}$ whose $e^-(E_1)=e(E_1)=n-1.$ 
Then $e^-(\dot{G}^{s,t})=e^-(\dot{G}_{U_1}^{s,t})+n-3=n-1+n-3=2n-4< \lfloor\frac{n-2}{2}\rfloor\lceil\frac{n-2}{2}\rceil+n-2$. Therefore,  Eq. (\ref{equation2}) holds if and only if $e^-(E_1)=n-2$ and $\sigma(uv)=+1,$ and $e^-(E_2)= \lfloor\frac{n-2}{2}\rfloor\lceil\frac{n-2}{2}\rceil$ 
if and only if $\dot{G}^{s,t}=\dot{G}_U^{\lfloor\frac{n-2}{2}\rfloor,\lceil\frac{n-2}{2}\rceil},$ where  $\dot{G}_U^{\lfloor\frac{n-2}{2}\rfloor,\lceil\frac{n-2}{2}\rceil}$ is obtained from $\dot{G}^{\lfloor\frac{n-2}{2}\rfloor,\lceil\frac{n-2}{2}\rceil}$
  by switching at vertex set $U=\{v,u_1,\dots,u_{\lfloor\frac{n-2}{2}\rfloor}\}.$
\hfill $\square$

\section{Proof of Theorem \ref{thm2}}
Let $\dot{G}$ be a  connected unbalanced signed graph of order  $n$,
and let  the unbalanced cycle $\mathcal{C}$ and the negative $uv$ in $\mathcal{C}$ be defined in Section 2.
 Based on Fact 2, by the table of the spectra of  signed graphs with at most six vertices \cite{B91}, we can check that Theorem \ref{thm2} is true for $n\le6.$ Therefore, we next  assume that 
 $n\ge 7.$
 
We first introduce some results about the largest eigenvalue of a graph $G$ or a signed graph $\dot{G}.$
The first lemma is \emph{Cauchy Interlacing Theorem}.
\begin{lemma}\label{Lem2.5}\cite{B11}
Let $A$ be a symmetric matrix of order $n$ with eigenvalues $\lambda_1\geq\lambda_2\geq\dots \geq\lambda_n$ and $B$ a principal submatrix of $A$ of order $m$ with eigenvalues $\mu_1\geq\mu_2\geq\dots \geq\mu_m.$ Then the eigenvalues of $B$ interlace the eigenvalues of $A,$ that is, $\lambda_i\geq \mu_i\geq \lambda_{n-m+i}$ for $i =1, \dots, m.$
\end{lemma}

The next lemma is  \emph{equitable partition}. Consider a partition
$\mathcal{P} = \{V_1,\dots , V_m\}$ of the set $V = \{1, \dots , n\}.$ The characteristic matrix $\chi_{\mathcal{P}}$ of $\mathcal{P}$ is the
$n\times m$ matrix whose columns are the characteristic vectors of $V_1,\dots , V_m$. Consider a symmetric
matrix $A$ of order $n,$ with rows and columns are partitioned according to $\mathcal{P}$. The partition of $A$ is
\emph{equitable} if each submatrix $A_{i,j}$ is formed by the rows of $V_i$ and the columns of $V_j$ has constant
row sums $q_{i,j}$. The $m \times m$ matrix $Q = (q_{i,j})_{1\le i,j\le m}$ is called the \emph{quotient} matrix of $A$ with respect
to equitable partition $\mathcal{P}$.

\begin{lemma}\label{lem2.5}\cite[p. 30]{B11}
The matrix A has the following two kinds of eigenvectors and eigenvalues:

$(i)$ The eigenvectors in the column space of $\chi_{\mathcal{P}}$; the corresponding eigenvalues coincide
with the eigenvalues of $Q.$

$(ii)$ The eigenvectors orthogonal to the columns of $\chi_{\mathcal{P}}$; the corresponding eigenvalues of $A$
remain unchanged if some scalar multiple of the all-one block $J$ is added to block $A_{i,j}$
for each $i, j \in \{1,\dots,m\}.$

Furthermore,  if $A$ is nonnegative and irreducible, then
$$\lambda_1(A) = \lambda_1(Q).$$
\end{lemma}

By Perron–Frobenius theorem, we can easily obtain the following.
\begin{lemma}\label{l3.3}
Let $G$ be a connected graph,  $u$ and $v$ be two non-adjacent vertices in $G$. Then 
$\lambda_1(G)<\lambda_1(G+uv).$  
\end{lemma}

The clique number of a graph $G$, denoted by $\omega(G)$, is the maximum order of a complete subgraph of the graph.
\begin{lemma}\cite{N02}\label{3.4}
Let $G$ be a graph with $e(G)$ edges and clique number $\omega(G)$. Then $$\lambda_1(G)\le \sqrt{2e(G)\frac{\omega(G)-1}{\omega(G)}}.$$
\end{lemma}

Sun, Liu and Lan \cite{S22} extended Lemma \ref{3.4} to signed graphs.
Let $\omega_b$ be the balanced clique number of $\dot{G}$, which is the maximum order of a balanced complete subgraph of $\dot{G}$.

\begin{lemma}\label{lem3.5}\cite[Theorem 1.3]{S22}
For a signed graph $\dot{G}$ with $e(\dot{G})$ edges and balanced clique number $\omega_b$, then
$$\lambda_1(\dot{G})\le \sqrt{2e(\dot{G})\frac{\omega_b-1}{\omega_b}}.$$
\end{lemma}

The following  lemma gives an upper bound of  the largest eigenvalue $\lambda_1({\dot{G}})$ of a signed graph $\dot{G},$ which is derived by Stani\'{c} \cite{S19}.

\begin{lemma}\cite[Theorem 3.1]{S19}\label{l3.2}
Every signed graph $\dot{G}$ contains a balanced spanning subgraph, say $\dot{H}$, which satisfies $\lambda_1({\dot{G}}) \le  \lambda_1(\dot{H})$.
\end{lemma}

\begin{remark}\label{r3.7}
Let $\textbf{x} = (x_1, x_2, \dots , x_n)^T$ be a unit eigenvector associated with $\lambda_1({\dot{G}}).$ Then there is a switching equivalent graph $\dot{G}_U$ of $\dot{G}$ such that 
 the coordinates of $\textbf{x}$ of $\dot{G}_U$ are non-negative. See \cite[Fact 2]{S19}.
By  the proof of \cite[Theorem 3.1]{S19}, the balanced spanning subgraph $\dot{H}$ in Lemma \ref{l3.2} is obtained  from $\dot{G}_U$ by removing all negative edges.
\end{remark}

Let $\textbf{0}_{m,n}$ (or $\textbf{0}$) denote the all-zeros matrix,  $j_n$ be an all-ones vector of order $n$ and $J_{m,n}$ (or $J$) denote the all-ones matrix.
We now give a descending order about the spectral radius of  signed  graphs  $\dot{G}^{s,t}$ ($s+t=n-2$).
Because of the symmetry, we assume that $t\ge s.$
\begin{lemma}\label{l3.6}
 $\lambda_1(\dot{G}^{1,n-3})=\frac{1}{2}( \sqrt{ n^2-8}+n-4)>\lambda_1(\dot{G}^{2,n-4})>\dots>\lambda_1(\dot{G}^{\lfloor\frac{n-2}{2}\rfloor,\lceil\frac{n-2}{2}\rceil})$.

\end{lemma}

\begin{proof}
The adjacency matrix $A_\sigma$ of $\dot{G}^{s,t}$ and its quotient matrix $Q^{s,t}$ are as following
\begin{gather*}
A_\sigma=
\begin{bmatrix}
0      & -1 &j_s& \textbf{0}    \\
-1 & 0& \textbf{0}&j_t \\
j_s^T & \textbf{0}&  (J-I)_s&J \\
\textbf{0}    &j_t^T &J&(J-I)_{t}
\end{bmatrix}~\text{and}~
Q^{s,t}=
\begin{bmatrix}
0      & -1 & s& 0     \\
-1 & 0&0&t \\
1 & 0& s-1&t \\
0     &1 & s&t-1
\end{bmatrix}, respectively.
\end{gather*}

By Lemma \ref{lem2.5}, then the eigenvalues of $Q^{s,t}$ are also the eigenvalues of $A_\sigma$
and the other eigenvalues of $A_\sigma$ remain the same if we add some scalar multiple of  $J$ from the blocks equal to $-1,$ $j_s,$ $j_t,$ $J$ or $J-I.$
Then  $A_{\sigma}$ and $Q^{s,t}$ become
\begin{gather*}
A^\prime=
\begin{bmatrix}
0      & 0 &\textbf{0}& \textbf{0}    \\
0 & 0& \textbf{0}&\textbf{0} \\
\textbf{0} & \textbf{0}&  -I_s&\textbf{0} \\
\textbf{0}    &\textbf{0}&\textbf{0}&-I_{t}
\end{bmatrix}~\text{and}~
Q^\prime=
\begin{bmatrix}
0      & 0 &0& 0    \\
0      & 0 &0& 0    \\
0      & 0 &-1& 0    \\
0      & 0 &0& -1
\end{bmatrix}, respectively.
\end{gather*}
The part of the spectrum of $A^\prime,$ which is not in the spectrum of $Q^\prime$ is $\{-1^{s+t-2}\}$. Thus the eigenvalues of $\dot{G}^{s,t}$ 
are the eigenvalues of $Q^{s,t}$ and  $-1$ with multiplicity $s+t-2.$
Therefore,
 $\lambda_1(\dot{G}^{s,t})=\lambda_1(Q^{s,t}).$
 
   By direct calculation, 
the characteristic polynomial of the matrix $Q^{s,t}$   is $$f(x,s,t)=x^4+ (-s-t+2)x^3+ (-2 s-2 t)x^2+ (2 s t-2)x+3 s t+s+t-1.$$
Then  $\lambda_1(Q^{1,n-3})$ is the largest root of $f(x,1,n-3)=(x^2-2) (6 - 2 n + (4 - n) x + x^2)=0.$
So,
 $$\lambda_1(\dot{G}^{1,n-3})=\lambda_1(Q^{1,n-3})=\frac{1}{2}( \sqrt{ n^2-8}+n-4).$$
 
Note that $$f(x,s-1,t+1)-f(x,s,t)=(2 x+3) (s-t-1).$$

As  $t\ge  s,$
then    $f(x,s-1,t+1)-f(x,s,t)< 0$ when $x\ge 0.$ 
Since $\lambda_1(\dot{G}^{s,t})$ is the largest root of $f(x,s,t)=0,$ it follows that 
$$\lambda_1(\dot{G}^{s-1,t+1})>\lambda_1(\dot{G}^{s,t}).$$ 
Hence,
 $\lambda_1(\dot{G}^{1,n-3})>\lambda_1(\dot{G}^{2,n-4})>\dots>\lambda_1(\dot{G}^{\lfloor\frac{n-2}{2}\rfloor,\lceil\frac{n-2}{2}\rceil})$.
\end{proof}
\begin{remark}\label{r3.9}
$\lambda_1(\dot{G}^{1,n-3})=\frac{1}{2}( \sqrt{ n^2-8}+n-4)>\frac{1}{2}(n-1+n-4)=n-\frac{5}{2}$ for $n\ge 7.$
\end{remark}

Before giving the proof of Theorem \ref{thm2}, we need the following three lemmas.

Let $H^{s,t}_1$ ($s+t=n-2$) be the underlying graph of  $\dot{G}^{s,t}-u_1v_1$.

\begin{lemma}\label{l3.7}
$ \lambda_1(H^{s,t}_1)<\lambda_1(\dot{G}^{1,n-3})$.
\end{lemma}

\begin{proof}
Because of the symmetry, we assume that $t\ge s.$

 If $s=1,$  giving a proper vertex partition $V(H^{1,n-3}_1)=\{u\}\cup \{v\}\cup \{u_1\}\cup \{v_1\} \cup N(v)\setminus\{v_1\},$ then 
  the adjacency matrix $A_1$ of graph $H^{1,n-3}_1$ and its quotient matrix $Q_1$ are
\begin{gather*}
A_1=
\begin{bmatrix}
0      & 1& 1 &   0&\textbf{0}\\
1 &0&0&1  &j_{n-4}\\
1 &0&0&0  &j_{n-4}\\
0&1  &0&0 &j_{n-4}\\
\textbf{0}    &j_{n-4}^T&j_{n-4}^T&j_{n-4}^T &(J-I)_{n-4} 
\end{bmatrix}~\text{and}~
Q_{1}=
\begin{bmatrix}
0      & 1& 1 &   0&0\\
1 &0&0&1  &n-4\\
1 &0&0&0  &n-4\\
0&1  &0&0 &n-4\\
0   &1&1&1 &n-5 
\end{bmatrix}.
\end{gather*}
By Lemma \ref{lem2.5}, we have $\lambda_1(H_1^{1,n-3})=\lambda_1(Q_1).$  By direct calculations,
 the characteristic polynomial of the matrix $Q_1$   
is
$$f_1(x)=  x^5+ (n-5) x^4+ (3 n-9) x^3+ (n-7) x^2+ (3 n-11) x+n-5,$$
and
 \begin{equation}\label{eq2}
\begin{split}f_{1}(\lambda_1(\dot{G}^{1,n-3}))&=-3 + \frac{33}{2}n - 8 n^2 + n^3 + \sqrt{n^2-8} ( \frac{31}{2} - 8 n + n^2)>0~~\text{for $n\ge 7$}.\end{split}
\end{equation}
 By Lemmas \ref{Lem2.5}, \ref{3.4} and Remark \ref{r3.9}, we have
 \begin{equation*}
 \begin{split}
 \lambda_2(H_1^{1,n-3})&\le \lambda_1(H_1^{1,n-3}-v)\le \sqrt{2e(H_1^{1,n-3}-v)\frac{\omega(H_1^{1,n-3}-v)-1}{\omega(H_1^{1,n-3}-v)}}\\ &=\sqrt{(n-2)(n-4)}< n-3<\lambda_1(\dot{G}^{1,n-3}).\end{split}
  \end{equation*} Thus from Eq. (\ref{eq2}), we have $\lambda_1(H_1^{1,n-3})<\lambda_1(\dot{G}^{1,n-3})$. 
  
  If $s\ge 2,$ 
giving a proper vertex partition $V(H^{s,t}_1)=\{u\}\cup \{v\}\cup \{u_1\}\cup \{v_1\} \cup N(u)\setminus\{u_1\}\cup N(v)\setminus\{v_1\}$, then
the adjacency matrix $A_1$ of graph $H^{s,t}_1$ and its quotient matrix $Q_1$ are
\begin{gather*}
A_1=
\begin{bmatrix}
0      & 1 &1& 0 &   j_{s-1}&\textbf{0}\\
1 & 0& 0&1 &\textbf{0} & j_{t-1}\\
1 &0&  0&0&j_{s-1}  &j_{t-1}\\
0 &1&  0&0&j_{s-1}  &j_{t-1}\\
j_{s-1}^T&\textbf{0}  &1&1&(J-I)_{s-1} &J\\
\textbf{0}    &j_{t-1}^T&1&1&J &(J-I)_{t-1} 
\end{bmatrix}~\text{and}~
Q_1=
\begin{bmatrix}
0      & 1 &1& 0 &   s-1&0\\
1 & 0& 0&1 &0 & t-1\\
1 &0&  0&0&s-1  &t-1\\
0 &1&  0&0&s-1  &t-1\\
1&0  &1&1&s-2 &t-1\\
0    &1&1&1&s-1 &t-2 
\end{bmatrix}.
\end{gather*}
By Lemma \ref{lem2.5}, we have $\lambda_1(H_1^{s,t})=\lambda_1(Q_1).$
Then the characteristic polynomial of the matrix $Q_1$   
is
$f_1(x,s,t)=x^6+ (4 -
     s - t) x^5+ (6 - 4 s - 4 t) x^4 + (2 s t-4 s - 4 t ) x^3+( 3st+ s + t-7) x^2+ (-4 + 4 s + 4 t - 2 s t) x+3-s-t.$

Note that $$f_1(x,s-1,t+1)-f_1(x,s,t)= (s-1-t)x(2 x^2+3x-2).$$
Since  $s\le t,$
then    $f_1(x,s-1,t+1)-f_1(x,s,t)< 0$ when $x\ge 1.$ 
Since $\lambda_1(H^{s,t}_1)$ is the largest root of $f_1(x,s,t)=0,$ it follows that 
$$\lambda_1(H_1^{s-1,t+1})>\lambda_1(H^{s,t}_1).$$ 
Hence,
 $\lambda_1(H_1^{2,n-4})>\dots>\lambda_1(H_1^{\lfloor\frac{n-2}{2}\rfloor,\lceil\frac{n-2}{2}\rceil}).$
 Lastly we prove that $\lambda_1(H_1^{2,n-4})<\lambda_1(\dot{G}^{1,n-3}).$
  By direct calculation, we have 
 \begin{equation}\label{eq3}
\begin{split}
&f_{1}(\lambda_1(\dot{G}^{1,n-3}),2,n-4)\\=&~-149 + 37 n + \frac{83}{2}n^2 - \frac{37}{2}n^3 + 2 n^4 + 
 \sqrt{n^2-8} (-32 +\frac{99}{2}n - \frac{37}{2} n^2 + 2 n^3)\\
>&~-149 + 37 n + \frac{83}{2}n^2 - \frac{37}{2}n^3 + 2 n^4 + 
(n-2) (-32 +\frac{99}{2}n - \frac{37}{2} n^2 + 2 n^3)\\
=&~-85 - 94 n + 128 n^2 - 41 n^3 + 4 n^4\\
>&~0~~~ \text{for $n\ge 7$}.\end{split}
\end{equation}
Note that $$\lambda_2(H_1^{2,n-4})\le\lambda_1(H_1^{2,n-4}-u_2)= \lambda_1(H_1^{1,n-4}) <\lambda_1(\dot{G}^{1,n-3}),$$
then from Eq. (\ref{eq3}), we have  $\lambda_1(H_1^{2,n-4})<\lambda_1(\dot{G}^{1,n-3})$.

This completes the proof.
\end{proof}
 Let $H^{s,t}_2$ ($s+t=n-2$ and $s\ge 2$) be the underlying graph of  $\dot{G}^{s,t}-u_1u_2$.
 
\begin{lemma}\label{l3.8}
$ \lambda_1(H^{s,t}_2)<\lambda_1(\dot{G}^{1,n-3})$.
\end{lemma}

\begin{proof}
If $s\ge 3,$ then
the adjacency matrix $A_2$ of graph $H_2^{s,t}$ and its quotient matrix $Q_2$ are
\begin{gather*}
A_2=
\begin{bmatrix}
0      & 1&1 & 1 &   j_{s-2}&\textbf{0}\\
1     & 0&0&0& \textbf{0}& j_{t}\\
1 &0&  0&0&j_{s-2}  &j_{t}\\
1    &0&  0&0&j_{s-2}  &j_{t}\\
j_{s-2}^T     &0  &j_{s-2}^T &j_{s-2}^T&(J-I)_{{s-2}} &J\\
0    &j_t^T&j_t^T&j_t^T&J &(J-I)_{t} 
\end{bmatrix}~\text{and}~
Q_2=
\begin{bmatrix}
0      & 1&1 & 1 &   s-2&0\\
1    & 0&0&0&0&t\\
1 &0&  0&0&s-2  &t\\
1    &0&  0&0&s-2  &t\\
1     &0  &1&1&s-3 &t\\
0    &1&1&1&s-2 &t-1 
\end{bmatrix}.
\end{gather*}
By Lemma \ref{lem2.5}, we have $\lambda_1(H_2^{s,t})=\lambda_1(Q_2).$
Then the characteristic polynomial of the matrix $Q_{2}$   
is
$f_2(x,s,t)= x^6+ (4 - s - t) x^5+ (6 - 4 s - 4 t) x^4+ (2 - 4 s - 
    4 t + 2 s t) x^3 + (-5 + s - 3 t + 3 s t) x^2+(-4 + 2 s + 4 t - 2 s t) x .$

Note that $$f_2(x,n-3,1)-f_2(x,s,n-s-2)=-(n-s-3) x (2 x-1) ((s-1) x+2s-4).$$
Since  $3\le s\le n-4,$
then    $f_2(x,n-3,1)-f_2(x,s,n-s-2)< 0$ when $x\ge1.$ 
Since $\lambda_1(H_2^{s,n-s-2})$ is the largest root of $f_2(x,s,n-s-2)=0,$ it follows that 
$$\lambda_1(H_2^{n-3,1})>\lambda_1(H_2^{s,n-s-2})~ \text{for}~3\le s\le n-4.$$ 
Moreover, we have 
 \begin{equation}\label{eq4}
\begin{split}
&f_2(\lambda_1(\dot{G}^{1,n-3}),n-3,1)\\=&~
\frac{1}{4}(4 n^3-32 n^2+64 n) \sqrt{n^2-8}+\frac{1}{4} (4 n^4-32 n^3+48 n^2+128 n-288)
\\
>&~\frac{1}{4}(4 n^3-32 n^2+64 n)(n-2)+\frac{1}{4} (4 n^4-32 n^3+48 n^2+128 n-288)\\
=&~2 (-36 + 22 n^2 - 9 n^3 + n^4)\\
>&~0~~\text{for $n\ge 7$}.\end{split}
\end{equation}
 By Lemmas \ref{Lem2.5}, \ref{3.4} and Remark \ref{r3.9}, we have
   \begin{equation*}
\begin{split}\lambda_2(H_2^{n-3,1})&\le \lambda_1(H_2^{n-3,1}-v_1)\le \sqrt{2e(H_2^{n-3,1}-v_1)\frac{\omega(H_2^{n-3,1}-v_1)-1}{\omega(H_2^{n-3,1}-v_1)}}\\&=\sqrt{(n-2)(n-4)}< n-3<\lambda_1(\dot{G}^{1,n-3}).
\end{split}
\end{equation*} Thus from Eq. (\ref{eq4}), we have $\lambda_1(H_2^{n-3,1})<\lambda_1(\dot{G}^{1,n-3})$. 
 
  If $s=2,$ then
the adjacency matrix $A_2$ of graph $H_2^{2,n-4}$ and its quotient matrix $Q_2$ are
\begin{gather*}
A_2=
\begin{bmatrix}
0      & 1&1 & 1 &   \textbf{0}\\
1     & 0&0&0&  j_{n-4}\\
1 &0&  0&0&j_{n-4}\\
1    &0&  0&0&j_{n-4}\\
0    &j_{n-4}^T&j_{n-4}^T&j_{n-4}^T&(J-I)_{n-4} 
\end{bmatrix}~\text{and}~
Q_2=
\begin{bmatrix}
0      & 1&1 & 1 &   0\\
1  & 0&0&0&{n-4}\\
1 &0&  0&0&{n-4}\\
1    &0&  0&0&{n-4}\\
0    &1&1&1&n-5
\end{bmatrix}.
\end{gather*}
By Lemma \ref{lem2.5}, then $\lambda_1(H_2^{s,t})=\lambda_1(Q_2).$
Then the characteristic polynomial of the matrix $Q_{2}$   
is
$f_2(x)=x^2 (-14 + 5 x + 4 x^2 + x^3 - n (-3 + 2 x + x^2)).$
And we have
 \begin{equation}\label{eq51}
\begin{split}
f_2(\lambda_1(\dot{G}^{1,n-3}))&=1/2(\lambda_1(\dot{G}^{1,n-3}))^2(9 n-38   - \sqrt{ n^2-8})\\
&>1/2(\lambda_1(\dot{G}^{1,n-3}))^2(9 n-38   - (n-2))\\
&=(\lambda_1(\dot{G}^{1,n-3}))^2(4n-18)\\
&>0.\end{split}
\end{equation}
By Lemma \ref{Lem2.5}, we have $$\lambda_2(H_2^{2,n-4})\le \lambda_1(H_2^{2,n-4}-v_1)= \lambda_1(H_2^{2,n-5})< \lambda_1(H_2^{3,n-5})< \lambda_1(H_2^{n-4,1})<\lambda_1(\dot{G}^{1,n-3}).$$
 Thus from Eq. (\ref{eq51}), we have $\lambda_1(H_2^{2,n-4})<\lambda_1(\dot{G}^{1,n-3})$. 
 
 This completes the proof.
\end{proof}

Let $H^{s^\prime,t^\prime}_3$ be the (unsigned) graph obtained from 
the underlying graph $G^{s^\prime,t^\prime}$ of 
$\dot{G}^{s^\prime,t^\prime}$  ($s^\prime+t^\prime=n-3$ and $s^\prime,t^\prime\ge 1$) and one isolated vertex $w$ by adding  positive edges $wu_1,\dots,wu_{s^\prime}$ and $wv_1,\dots,wv_{t^\prime}.$ 
 
\begin{lemma}\label{l3.9}
$ \lambda_1(H^{s^\prime,t^\prime}_3)<\lambda_1(\dot{G}^{1,n-3})$.
\end{lemma}

\begin{proof}
Because of the symmetry, we assume that $t^\prime\ge s^\prime.$
Then
the adjacency matrix $A_3$ of graph $H^{s^\prime,t^\prime}_3$ and its quotient matrix $Q_3$ are
\begin{gather*}
A_3=
\begin{bmatrix}
0      & 1&j_{s^\prime} &    \textbf{0}&0\\
1     & 0&  \textbf{0}&  j_{t^\prime} &0\\
j_{s^\prime}^T &\textbf{0}&  (J-I)_{s^\prime}&J&j_{s^\prime}  \\
\textbf{0}    &j_{t^\prime}^T&  J& (J-I)_{t^\prime}&j_{t^\prime}  \\
\textbf{0}    &\textbf{0}  &j_{s^\prime}^T &j_{t^\prime}^T&0
\end{bmatrix}~\text{and}~
Q_3=
\begin{bmatrix}
0      & 1&s^\prime &   0&0\\
1    & 0&0& t^\prime &0\\
1 &0&  s^\prime-1& t^\prime&1 \\
0  &1&  s^\prime& t^\prime-1&1  \\
0   &0  &s^\prime&t^\prime&0
\end{bmatrix}.
\end{gather*}
By Lemma \ref{lem2.5}, we have $\lambda_1(H_3^{s^\prime,t^\prime})=\lambda_1(Q_3).$
Then the characteristic polynomial of the matrix $Q_{3}$   
is
$f_3(x,s^\prime,t^\prime)=s^\prime + t^\prime - 2 s^\prime t^\prime - (1 - 2 s^\prime - 2 t^\prime - s^\prime t^\prime) x - (2 + s^\prime + t^\prime - 
    2 s^\prime t^\prime) x^2 - (3 s^\prime + 3 t^\prime) x^3 - (-2 + s^\prime + t^\prime) x^4 + x^5.$

Note that $$f_3(x,s^\prime-1,t^\prime+1)-f_3(x,s^\prime,t^\prime)=(s^\prime-1  - t^\prime) (2 x^2 + x-2).$$
Since  $s^\prime \le t^\prime,$
then    $f_3(x,s^\prime-1,t^\prime+1)-f_3(x,s^\prime,t^\prime)< 0$ when $x\ge 1.$ 
Since $\lambda_1(H_3^{s^\prime,t^\prime})$ is the largest root of $f_3(x,s^\prime,t^\prime)=0,$ it follows that 
$$\lambda_1(H_3^{s^\prime-1,t^\prime+1})>\lambda_1(H_3^{s^\prime,t^\prime}).$$ 
Hence,
 $\lambda_1(H_3^{1,n-4})>\dots>\lambda_1(H_3^{\lfloor\frac{n-3}{2}\rfloor,\lceil\frac{n-3}{2}\rceil}).$
Note that the graph $H_3^{1,n-4}$ is isomorphic to the graph $H_1^{1,n-3}.$ So
 $\lambda_1(H_3^{1,n-4})<\lambda_1(\dot{G}^{1,n-3})$ by Lemma \ref{l3.7}. 
  
   This completes the proof.
\end{proof}
By Lemmas \ref{l3.7}, \ref{l3.8} and \ref{l3.9}, we have 

\begin{corollary}\label{c3.9}
Let $G$ be a connected (unsigned) graph of order $n,$
and let  $u$ and $v$ be two  vertices in $G$ such that  $u$ and $v$ have no common neighbor in $G.$ Suppose that there have two vertices $u_1$ and $v_1$ in $G$ such that $u\sim u_1$ and $v\sim v_1.$

$(i)$  If there is a vertex $w$ that is  adjacent to neither $u$ nor $v$, then $\lambda_1(G)< \lambda_1(\dot{G}^{1,n-3}).$

$(ii)$ If  $G[N(u)\cup N(v)]$ is not a clique, then $\lambda_1(G)< \lambda_1(\dot{G}^{1,n-3}).$

\end{corollary}

\begin{proof}
 If there is a vertex $w$ that is  adjacent to neither $u$ nor $v,$ then $G$ is a spanning subgraph of $H_3^{s^\prime,t^\prime}.$ By Lemma \ref{l3.3}, then
$\lambda_1(G)\le   \lambda_1(H_3^{s^\prime,t^\prime})$. 
 If  $G[N(u)\cup N(v)]$ is not a clique, then  $u_{i_1}\not\sim u_{i_2}$, or $v_{j_1}\not\sim v_{j_2}$, or $u_{i}\not\sim v_{j}$ for some vertices $u_{i_1}, u_{i_2},u_{i}$ of $N(u)\setminus\{v\}$ and 
 $v_{j_1}, v_{j_2}, v_{j}$ of $N(v)\setminus\{u\}.$ Then $G$ is a spanning subgraph of $H_1^{s,t}$ or $H_2^{s,t}.$
 By Lemma \ref{l3.3}, then
$ \lambda_1(G)\le \lambda_1(H_i^{s,t})$ for $i=1$ or 2.
Now by Lemmas  \ref{l3.7}, \ref{l3.8} and  \ref{l3.9}, we have $\lambda_1(G)<\lambda_1(\dot{G}^{1,n-3}).$
\end{proof}

 The  \emph{negation} of  $\dot{G}$  (denoted by $-\dot{G}$) is  obtained  by reversing the sign of every edge  in $\dot{G}$.
Obviously, the eigenvalues of $-\dot{G}$ are obtained by reversing the sign of the eigenvalues of $\dot{G}.$ 

\noindent\emph{Proof of Theorem \ref{thm2}:}
Suppose to the contrary that there is  a connected $C_3^-$-free signed graph $\dot{G}$  that is  not switching equivalent to $\dot{G}^{1,n-3}$ and
$\rho(\dot{G})\ge \lambda_1(\dot{G}^{1,n-3}).$

Recall that $\rho(\dot{G})=max\{\lambda_1(\dot{G}),-\lambda_n(\dot{G})\}.$ If
$\lambda_1(\dot{G})<-\lambda_n(\dot{G}),$ then
 $\rho(\dot{G})=-\lambda_n(\dot{G})$ and we consider the negation of $\dot{G}$ instead.
For convenience, let $\dot{G}_1=-\dot{G}.$ Then $\rho(\dot{G})=-\lambda_n(\dot{G})=\lambda_1(\dot{G}_1).$
Since $\dot{G}$ is $C_3^-$-free, then  $\dot{G}_1$ is $C_3^+$-free.
Therefore, $\omega_b=2.$ By Lemma \ref{lem3.5}, we have 
$\lambda_1(\dot{G}_1)\le \sqrt{e(\dot{G})}.$
By Theorem \ref{thm1},
 we know that $e(\dot{G}_1)=e(\dot{G})\le \frac{n(n-1)}{2}-n+2.$
 Then 
  \begin{equation*}\label{eq5}
\begin{split}
\lambda_1(\dot{G}_1)&\le \sqrt{\frac{n(n-1)}{2}-n+2}= \sqrt{\frac{(n-1)(n-2)}{2}+1}\\&<  \sqrt{\frac{(n-1)(n-2)+n^2-7n+\frac{21}{2}}{2}}=n-\frac{5}{2}.\end{split}
\end{equation*}
 By Remark \ref{r3.9}, then $\lambda_1(\dot{G}_1)=\rho(\dot{G})
<\lambda_1(\dot{G}^{1,n-3})$,  which contradicts to the hypothesis. Next we just consider that $\lambda_1(\dot{G})\ge -\lambda_n(\dot{G})$
and let $\rho(\dot{G})=\lambda_1(\dot{G}).$ 
We divide into two cases: $N(u)\cap N(v)=\emptyset$ and $N(u)\cap N(v)\ne \emptyset$.

\textbf{Case 1.} $N(u)\cap N(v)=\emptyset.$
 Up to switching equivalence, we can assume that $\sigma(vv_j)=+1$ for all $j=1,\dots,t.$ 
By Corollary \ref{c3.9}, we get that each vertex in $\dot{G}$ is adjacent to either $u$ or $v$ and $\dot{G}[V(\dot{G})\setminus \{u,v\}]$ is a clique, otherwise 
$\lambda_1(\dot{G})\le  \lambda_1(G)< \lambda_1(\dot{G}^{1,n-3})$, which contradicts to the hypothesis. 
Since $\dot{G}$ is unbalanced, then  
 $\dot{G}[V(\dot{G})\setminus \{u,v\}]$ has at least one positive edge, say $u_1v_1.$
If $\dot{G}[V(\dot{G})\setminus \{u,v\}]$  has a negative edge, say $u_iv_j,$ 
since $u_1\mathop{\sim}\limits^{+} u_i$ and  $v_1\mathop{\sim}\limits^{+} v_j$ (by Fact 1),
then $\{u_1,u_i,v_j\}$ or $\{v_1,u_i,v_j\}$ induces a negative triangle $C_3^-,$ which is a 
contradiction.
Therefore, $\dot{G}[V(\dot{G})\setminus \{u,v\}]$ is an all-positive clique and 
$\dot{G}\sim \dot{G}^{s,t}.$ By Lemma \ref{l3.6}, we have $\lambda_1(\dot{G})\le \lambda_1(\dot{G}^{1,n-3}),$ with equality holding if and only if $\dot{G}\sim \dot{G}^{1,n-3}.$

\textbf{Case 2.} $N(u)\cap N(v)\ne\emptyset$.
Suppose that  $N(u)\cap N(v)=W=\{w_1,\dots,w_k\}$ $(k\ge 1)$ and
 $N[u]\setminus N[v]=\{u_1,\dots,u_{s_1}\}$ ($s_1+k=s$) and $N[v]\setminus N[u]=\{v_1,\dots,v_{t_1}\}$ ($t_1+k=t$), where $s_1\ge 1$ and $t_1\ge 1$ by Lemma \ref{2}.
Since $\dot{G}$ is  $C_3^-$-free, then $\sigma(vw_i)=-1$ for  $i=1,\dots,k.$ Up to switching equivalence, we can let $\sigma(vv_i)=-1$ for 
each vertex $v_i\in N[v]\setminus N[u].$ Now switching at vertex $v,$ 
then all edges associated to the vertex $v$ are positive. 
 Without loss of generality, we assume that 
 $\mathcal{C}=uu_1v_1vu$ (if the length of $\mathcal{C}$ is 4) or $\mathcal{C}=uu_1z_1\dots z_q v_1vu$ (if the length of $\mathcal{C}$ is greater than 4).

\noindent\textbf{Claim 1.}
Two vertices $u_1$ and $v_1$ have no common neighbor in $\{w_1,\dots,w_k\}.$

\begin{proof}
If $u_1\mathop{\sim}\limits^{-}  v_1$, then $u_1$ and $v_1$ have no common neighbor in $\{w_1,\dots,w_k\}$ as $\dot{G}$ is $C_3^-$-free.
Otherwise $u_1\mathop{\sim}\limits^{+}  v_1$ or $u_1\not\sim v_1$.
If $u_1$ and $v_1$ have a common neighbor in $\{w_1,\dots,w_k\},$ say $w_1,$
then we can find an  unbalanced cycle $\mathcal{C}^\prime=w_1u_1z_1\dots z_q v_1w_1$ with fewer length, which 
contradicts to the choice of $\mathcal{C}$.
 This completes the proof.
\end{proof}

Let $W=W_0\cup W_1\cup W_2$, where each vertex in $W_0$  is adjacent to neither $u_1$ nor $v_1,$   each vertex in $W_1$ (resp.  $W_2$) is adjacent to $u_1$ (resp. $v_1$) and $|W_i|=k_i$ for $i=0, 1$ or 2. 
By Claim 1, we have $W_1\cap W_2=\emptyset$.
If $t_1=1$ and $k_2=0,$ then  $N(v)\cap N(v_1)=\emptyset$,  it is  similar to Case 1.
Then $t_1\ge 2$ or $k_2\ge 1.$ Similarly, $s_1\ge 2$ or $k_1\ge 1.$

Let $V(\dot{G})\setminus (\{u,v\}\cup N(u)\cup N(v))=Z=\{z_1,\dots,z_p\}$ and
let $\textbf{x} = (x_u, x_v,x_{u_1},\dots,x_{u_{s_1}},x_{v_1},$ $\dots,x_{v_{t_1}},x_{w_1}, \dots, x_{w_k},x_{z_1}, \dots, x_{z_p})^T$ be a unit eigenvector associated with $\lambda_1({\dot{G}})$
and  $U=\{i|x_i<0 \}$ be the vertex subset of $V(\dot{G}).$
Then $\dot{G}_U=(\dot{G},\sigma_U)$ is obtained from $\dot{G}$ by switching at vertex subset  $U.$ 
By Lemma \ref{l3.2} and Remark \ref{r3.7}, then 
  $\lambda_1(\dot{G})=\lambda_1(\dot{G}_U)\le \lambda_1(\dot{H}),$ where $\dot{H}$
is  a balanced spanning subgraph of $\dot{G}_U$ by removing all negative edges.
By Corollary \ref{c3.9}, we have  three claims.

\noindent\textbf{Claim 2.}  $x_v<0$ if and only if $x_u<0$.
\begin{proof}
If $x_v<0,$ then $v\in U$.
For a contradiction, assume that $x_u\ge 0.$ Then 
 $\sigma_U(uw_i)\sigma_U(vw_i)=-1$ for $i=1,\dots,k$ and exactly one edge of $uw_i$ and $vw_i$ (in $\dot{G}_U$) is negative. So $u$ and $v$ have no common neighbor in $\dot{H}.$  
 Let  $\dot{H}^\prime$ be $\dot{H}$ or $\dot{H}+vv_1+uu_1$ (if $v\not\sim v_1$ or $u\not\sim u_1$ in $\dot{H}$) where $\sigma(vv_1)=\sigma(uu_1)=+1.$
Since $w_i\not\sim v_1$ or $w_i\not\sim u_1,$ then $\dot{H}^\prime[N(u)\cup N(v)]$ is not a clique. By Lemma \ref{l3.3} and Corollary \ref{c3.9}, then   $\lambda_1(\dot{G}) \le  \lambda_1(\dot{H})\le \lambda_1(\dot{H}^\prime)< \lambda_1(\dot{G}^{1,n-3})$, which contradicts to the hypothesis. So $x_u<0.$
Because of the symmetry, if $x_u<0,$ then $x_v<0.$
\end{proof}

\noindent\textbf{Claim 3.} $x_{v_1}<0$ if and only if $x_{v}<0$.
\begin{proof}  If $x_{v_1}<0,$ then $v_1\in U$.
For a contradiction, assume that $x_v\ge 0.$
Then
 $\sigma_U(vw)\sigma_U(v_1w)=-1$ for all vertices $w\in N(v)\cap N(v_1).$ So  $v$ and $v_1$ have no common neighbor in $\dot{H}.$ 
  Let  $\dot{H}^\prime$ be $\dot{H}$ or $\dot{H}+u_1v_1$ (if $u_1\not\sim v_1$ in $\dot{H}$) where $\sigma(u_1v_1)=+1.$
Since $u\not\sim v_2$ (if $t_1\ge 2$) or $u_1\not\sim w_i$ for one $i$ (if $k_2\ge 1$), then $\dot{H}^\prime[N(v)\cup N(v_1)]$ is not a clique. By Lemma \ref{l3.3} and  Corollary \ref{c3.9}, then   $\lambda_1(\dot{G}) \le  \lambda_1(\dot{H})\le  \lambda_1(\dot{H}^\prime)< \lambda_1(\dot{G}^{1,n-3})$, which contradicts to the hypothesis. So $x_v<0.$
Similarly, by Claim 2 and Corollary \ref{c3.9}, we can prove that if $x_{v}<0,$ then $x_{v_1}<0.$
\end{proof}

\noindent\textbf{Claim 4.} $x_{u_1}<0$ if and only if $x_{u}<0$.
\begin{proof}
The proof is similar to Claim 3.
\end{proof}

If $N(u_1)\cap N(v_1)=\emptyset$,  it is  similar to Case 1. Next  suppose that 
$N(u_1)\cap N(v_1)=\{y_1,\dots,y_\ell\}$ ($\ell\ge 1$). Then $y_1\in N(v)\cup N(u)\cup Z.$

\textbf{Subcase 2.1.} $\sigma(u_1v_1)=-1.$  By Claims 2, 3 and 4, then $x_{u}< 0$ if and only if 
$x_{v}< 0$ if and only if $x_{u_1}< 0$ if and only if $x_{v_1}< 0$.
 Then there are two subsubcases.

\textbf{Subsubcase 2.1.1.}  $x_{v_1}\ge0$ and $x_{u_1}\ge 0.$  Then $x_u\ge 0$ and $x_v\ge 0$.
Since $\dot{G}$ is $C_3^-$-free, we have $\sigma(u_1y_i)\sigma(v_1y_i)=-1$ for all $i=1,\dots,\ell.$ Then exactly one edge of $u_1y_i$
and $v_1y_i$ is negative.   So $u_1$ and $v_1$ has no common neighbor in $\dot{H}.$
Since $y_1\not\sim v$ (if $y_1\in N(u)\cup Z$) and $y_1\not\sim u$ (if $y_1\in N(v)$), then $\dot{H}[N(u_1)\cup N(v_1)]$ is not a clique. By Lemma \ref{l3.2} and Corollary \ref{c3.9}, then $\lambda_1(\dot{G}) \le  \lambda_1(\dot{H})< \lambda_1(\dot{G}^{1,n-3})$, which contradicts to the hypothesis.

\textbf{Subsubcase 2.1.2.}  $x_{u_1}<0$ and $x_{v_1}< 0.$  Then $x_u<0$ and $x_v<0$.
Switching at four vertices $u,v,u_1$ and $v_1,$  then it is  similar to Subsubcase 2.1.1.

 \textbf{Subcase 2.2.} $\sigma(u_1v_1)=+1$ or 0.
  Without loss of generality, we may assume that $\sigma(u_iv_j)\ne -1$ for all $i=1,\dots,s_1$ and $j=1,\dots,t_1.$
Then the length of $\mathcal{C}$ is at least 5.
 If $\sigma(u_1v_1)=+1$, then 
 we can find an unbalanced triangle $\mathcal{C}^\prime=u_1z_1v_1u_1$ (if $q=1$) or  an  unbalanced cycle $\mathcal{C}^\prime=u_1z_1\dots z_qv_1u_1$ (if $q\ge 2$) in $\dot{G}$ with fewer length, which 
contradicts to the choice of $\mathcal{C}$. 
Therefore, $\sigma(u_1v_1)=0$. 
If the length of  $\mathcal{C}$ is greater than 5, then $N(u_1)\cap N(v_1)=\emptyset$. By Corollary \ref{c3.9} again,  we have $\lambda_1(\dot{G})\le\lambda_1(G)< \lambda_1(\dot{G}^{1,n-3})$, which contradicts to the hypothesis.
Therefore, the length of  $\mathcal{C}$ is 5.
Suppose that $\mathcal{C}=uu_1y_1v_1vu$ where $u_1 \mathop{\sim}\limits^{+} y_1$ and $v_1 \mathop{\sim}\limits^{-} y_1$.  
Then  there is no vertex $w$ with $\sigma(u_1w)\sigma(v_1w)=+1,$ otherwise we can find  an  unbalanced cycle $\mathcal{C}^\prime=u_1y_1v_1wu_1$ with fewer length, which 
contradicts to the choice of $\mathcal{C}$.  So $\sigma(u_1y_i)\sigma(v_1y_i)=-1$ for all vertices $y_i \in N(u_1)\cap N(v_1)$.
Similar to Subcase 2.1, by Claims 2, 3 and 4, there are two subsubcases.

\textbf{Subsubcase 2.2.1.}  $x_{v_1}\ge0$ and $x_{u_1}\ge 0.$  Then $x_u\ge 0$ and $x_v\ge 0$,
  $u_1$ and $v_1$ has no common neighbor in $\dot{H}.$
Since $y_1\not\sim v$ and $y_1\not\sim u$, then $\dot{H}[N(u_1)\cup N(v_1)]$ is not a clique. By Lemma \ref{l3.2} and Corollary \ref{c3.9}, then $\lambda_1(\dot{G}) \le  \lambda_1(\dot{H})< \lambda_1(\dot{G}^{1,n-3})$, which contradicts to the hypothesis.

\textbf{Subsubcase 2.2.2.}  $x_{u_1}<0$ and $x_{v_1}< 0.$ Then $x_u<0$ and $x_v<0$.
Switching at four vertices $u,v,u_1$ and $v_1,$   then it is  similar to Subsubcase 2.2.1.

We complete the proof.
\hfill $\square$

\vskip 0.4 true cm
\begin{center}{\textbf{Acknowledgments}}
\end{center}

This project  is supported by the National Natural Science Foundation of China (No.11971164, 12001185, 12101557) and Zhejiang Provincial Natural Science Foundation of China (LQ21A010004).

\end{document}